\documentclass[a4paper]{amsart}
\usepackage{amsmath,caption,booktabs,lipsum}
\numberwithin{equation}{section}
\usepackage{amsthm}
\usepackage{amsfonts}
\usepackage{amssymb}
\usepackage{mathrsfs}
\usepackage{tikz}
\usepackage{environ}
\usepackage{elocalloc}
\usepackage{url}
\usepackage{caption}
\usepackage{graphicx}
\usepackage{CJKutf8}
\usepackage[normalem]{ulem}
\usepackage{xcolor}
\usepackage{etoolbox}
\usepackage{caption}
\usepackage{mathtools}
\usepackage{dcpic}
\usepackage{tikz-cd}
\usepackage[bottom]{footmisc}
\usepackage{hyperref}
\usepackage{enumitem}
\usepackage{stmaryrd}
\usetikzlibrary{positioning}
\usetikzlibrary{arrows,chains,positioning,scopes,quotes}
\usetikzlibrary{decorations.markings}

\newcommand{\ai}{\alpha}
\newcommand{\be}{\beta}
\newcommand{\Ga}{\Gamma}
\newcommand{\ga}{\gamma}

\newcommand{\e}{\epsilon}

\newcommand{\lam}{\lambda}

\newcommand{\om}{\omega}
\newcommand{\Om}{\Omega}
\newcommand{\fr}{\textnormal{FocalRad}}
\newcommand{\si}{\sigma}
\newcommand{\Si}{\Sigma}

\newcommand{\cms}{\operatorname{comass}}

\newcommand{\ms}{\mathbf{M}}

\newcommand{\cd}{\cdots}
\newcommand{\T}{\mathbf{T}}
\newcommand{\s}{\subset}

\newcommand{\cp}{^\complement}

\newcommand{\cpt}{\mathbb{CP}}

\newcommand{\ti}[1]{\tilde{#1}}

\newcommand{\ov}[1]{\overline{#1}}
\newcommand{\no}[1]{\left\lVert#1\right\rVert}
\newcommand{\cu}[1]{\left\llbracket#1\right\rrbracket}

\newcommand{\du}{^\ast}

\newcommand{\m}{^{-1}}
\newcommand{\ts}{\otimes}

\newcommand{\w}{\wedge}

\newcommand{\pd}{\partial}

\newcommand{\R}{\mathbb{R}}
\newcommand{\Z}{\mathbb{Z}}

\newcommand{\C}{\mathbb{C}}

\newcommand{\mina}{\min\ms}

\makeatother
\theoremstyle{plain}
\newtheorem{thm}{Theorem}[section]
\newtheorem{lem}[thm]{Lemma}

\theoremstyle{definition}
\newtheorem{defn}{Definition}[section]

\theoremstyle{remark}
\newtheorem{rem}{Remark}

\title[Non-calibratable area-minimizers]{Homologically area-minimizing surfaces that cannot be calibrated: infinite Lavrentiev gap}
\author{Zhenhua Liu}
\dedicatory{Dedicated to Xunjing Wei}
\begin{document}
	\maketitle\vspace{-3em}
	\begin{abstract}
In 1974, Federer \cite{HF2} proved that all area-minimizing hypersurfaces on orientable manifolds were calibrated by weakly closed differential forms. However, in this manuscript, we prove the contrary in higher codimensions: calibrated area-minimizers are non-generic. This is surprising given that almost all known examples of area-minimizing surfaces are confirmed to be minimizing via calibration.

Let integers $d\ge 1$ and $c\ge 2$ denote dimensions and codimensions, respectively. Let $M^{d+c}$ denote a closed, orientable, smooth manifold of dimension $d+c$. For each $d$-dimensional integral homology class $[\Si]$ on $M$, we introduce $\Om_{[\Si]}$ as the set of metrics for which any $d$-dimensional homologically area-minimizing surface in the homology class $[\Si]$ in any $g\in \Om_{[\Si]}$ cannot be calibrated by any weakly closed measurable differential form. Our main result is that $\Om_{[\Si]}$ is always a non-empty open set.

To exemplify the prevalence of such phenomenon, we show that for any homology class $[\Si]$ on $\mathbb{CP}^n,$ the closure of $\Om_{[\Si]}$ contains the Fubini-Study metric. 

In the hypersurface case, we show that even when a smooth area-minimizer is present, the calibration forms are compelled in some cases to have a non-empty singular set. This provides an answer to a question posed by Michael Freedman \cite{MF1}.

The above phenomenon is essentially due to the Lavrentiev phenomenon of the minimal mass of homology classes. In this direction, for $[\Si]$ with $d\ge1 ,c\ge 2$, we show that the ratio of the integral minimal mass to the real minimal mass is unbounded when we consider all Riemannian metrics. Also, it is always possible to fill a multiple of a homology class with an arbitrarily small area compared to the class itself. This settles the Riemannian version of several conjectures by Frank Morgan \cite{FM4}, Brian White \cite{BW} and Robert Young \cite{RY}.
	\end{abstract}
	\tableofcontents
\section{Introduction}
In this paper, area minimizing surfaces refer to area-minimizing integral currents, which roughly speaking are oriented surfaces counted with multiplicity, minimizing the area functional. Calling them surfaces is justified thanks to the Almgren's Big Theorem (\cite{A}) and De Lellis school's work (\cite{DS1}\cite{DS2}\cite{DS3}\cite{DMS}). Their results show that $n$-dimensional area minimizing integral currents are smooth manifolds outside of a rectifiable singular set of dimension at most $n-2.$ (In the codimension 1 case, the dimension of the singular set can be reduced to $n-7$ by \cite{HF1}.) 

On the other hand, in the foundational work \cite{FF}, it is established on a Riemannian manifold, any integral homology class admit a representative that is area-minimizing. Consequently, we have extremely powerful existence and regularity at our disposal. 

Contrarily, proving that a selected representative of a homology class is area-minimizing poses substantial challenges. To the author's knowledge, over the past 70 years, calibrations (\cite{HLg}) essentially remain the sole method available, encompassing K\"{a}hler, special Lagrangian, associative, and Caley calibrations in specific holonomic geometries, Lawlor's vanishing calibrations \cite{GL}, Morgan and Mackenzie's planar calibrations \cite{FM}\cite{DN} and the foliation by minimal or mean convex/concave hypersurfaces in hypersurface instances.

Furthermore, Federer proved in \cite{HF2} that every area-minimizing hypersurface residing on compact orientable manifolds can be calibrated by a weakly closed differential form. 

In light of mounting evidences, a natural question arises: 

\textbf{Are calibrated area-minimizers also generic in higher codimensions, given that calibration is predominantly the only way for proving area-minimization?} 

There exist sporadic counterexamples. Torsion classes, for instance, are trivially non-calibratable. For non-torsion classes, Section 5.11 of \cite{HF2} gives a $3$-torus counter-example. Nonetheless, it remains an open question whether such non-calibratable minimizers are a general occurrence or happening only on special manifolds.

To resolve the above question, we give a negative answer as follows.
\begin{thm}\label{noncal}Assume that,
	\begin{itemize}
\item $M^{d+c}$ is a compact closed orientable $d+c$ dimensional smooth manifold with $d\ge 2,c\ge 1.$
\item $[\Si]$ is a $d$-dimensional integral homology class on $M$.
	\end{itemize}
Define  $\Om_{[\Si]}$ to be the set of metrics so that that for any metric $g\in \Om_{[\Si]}$, every $d$-dimensional homologically area-minimizing integral current in the homology class $[\Si]$ cannot be calibrated by real flat cochains. Then
\begin{itemize}
	\item $\Om_{\Si}$ is open.
	\item  $\Om_{\Si}$ is non-empty.
\end{itemize}
\end{thm}  
\begin{rem}
	By 4.4.19 in \cite{FF} and Section 4.6 in \cite{HF2}, any bounded weakly closed measurable differential form is a real flat cochain and vice versa.
\end{rem}
\begin{rem}
	The proof of the above theorem gives abundance of area-minimizing integral currents of which their multiples are not area-minimizing, even on manifolds with no torsion classes at all.
\end{rem}
\begin{rem}
	In view of the above theorem, Camillo De Lellis conjectures that $\Om_{[\Si]}$ is dense.
\end{rem}
One could postulate, in opposition to Theorem \ref{noncal}, that the $\Om_{[\Si]}$ metrics are hand-tailored to the problem, presupposing that proximity to standard classical metrics would imply calibration of area-minimizers. Regrettably, such an assumption is unfounded.
\begin{thm}\label{cpj}
	For $c,d\ge 1,$	Let $[\Si]$ be a $d$-dimensional integral homology class on a $\frac{d+c}{2}$-dimensional complex projective space $\mathbb{CP}^{\frac{d+c}{2}}.$ Then in the notation of Theorem \ref{noncal}, the Fubini-Study metric is in the closure of $\Om_{[\Si]}$. 
\end{thm}
\begin{rem}
	This shows that in higher codimensions, metrics with non-calibratable minimizers are not artificial constructions but indeed very natural.
\end{rem}
In codimension $1$, we always have a calibration by \cite{HF2}. Professor Michael Freedman raised the following question in personal communication (\cite{MF}), to paraphrase him, 

\textbf{When there is a smooth area-minimizing hypersurface, what is the best regularity of calibration forms?} 

We show that the answer might not be as optimistic as one hopes in Theorem A.1 of \cite{MF}.
\begin{thm}\label{singcal}
	Let $d\ge 7$ be integers and $M^{d+1}$ be a compact smooth orientable manifold. For every $d$-dimensional integral homology class $[\Si]$, there exists a smooth metric $g$ so that
	\begin{itemize}
		\item there are smooth area-minimizing hypersurfaces in $[\Si]$ in the metric $g$,
		\item any flat cochain calibrating $[\Si]$ must have a singular set with positive (possibly infinite) Hausdorff  $d-7$ measure.
	\end{itemize} 
\end{thm}
Now we are ready to delve deeper into the ideas behind Theorem \ref{noncal}. The non-calibrated minimizers is closely connected to the Lavrentiev phenomenon of minimal area, and is a corollary of the following results. First we need some notions as follows.
\begin{defn}
	For an integral current (or an integral homology class) $T$ on a smooth Riemannian manifold $(M,g)$, define the real and integral minimal mass as
	\begin{align*}
		\mina_\R^g([T])=\inf_{S\text{ homologous to }T\text{ over }\R} \ms(S)\\
		\mina_\Z^g([T])=\inf_{S\text{ homologous to }T\text{ over }\Z} \ms(S).
	\end{align*}
\end{defn}
\begin{rem}
	Here homologous over $\Z$ (or $\R$) means an integral current (or real flat chain) homologous in $\Z$ (or $\R$) coefficient, respectively. See Section \ref{cfc}. Also note that $\mina$ depends only on the homology class of $T$ if $\pd T=0.$
\end{rem} Federer proved in \cite{HF2} that for any integer $k,$ we have
 \begin{align*}
 	\mina_\R^g(k[T])=k\mina_\R^g([T]),\\
 	\mina_\R^g([T])\le \mina_\Z^g([T]),\\
 	\lim_{k\to\infty}\frac{\mina_\Z^g(k[T])}{k}=\mina_\R^g([T]).
 \end{align*}
Moreover, by \cite{HF2}, $T$ is calibrated if and only if  $\mina_\R^g([T])= \mina_\Z^g([T]).$

Frank Morgan (\cite{FM4}), Brian White (\cite{BW}) and Robert Young (\cite{RY}) have investigated the Lavrentiev phenomenon of \begin{align}\label{lavin}
	\frac{\mina_\Z^g([T])}{\mina_\R^g([T])}>1,\frac{k\mina_\Z^g([T])}{\mina_\Z^g(k[T])}>1.
\end{align} They have all raised the following two questions (Problem 1.13 in \cite{GMT}), to rephrase them,
\begin{itemize}
	\item \textbf{Is ${\mina_\Z^g([\Si])}/{\mina_\R^g([\Si])}$ bounded?} 
	\item \textbf{Is it possible that $\mina_\Z^g(k[T])<\mina_\Z^g([T])$?}
\end{itemize}
In other words, how large can the Lavrentiev gaps be in the two inequalities in (\ref{lavin}). We show that both gaps are unbounded in the homology setting.
\begin{thm}\label{inflavr}With the same assumption of Theorem \ref{noncal}, 
	\begin{align*}
			\sup_{g}\frac{\mina_\Z^g([\Si])}{\mina_\R^g([\Si])}=\infty,
	\end{align*}with $g$ ranging over all smooth Riemannian metrics.
\end{thm}
\begin{rem}
	Note that the above ratio is independent of constant scalings of the metric.
\end{rem}
This follows from a much more refined result.
\begin{thm}\label{inflav}
With the same assumption of Theorem \ref{noncal}, if $[\Si]$ is a non-torsion class, then there exists a sequence of integers $k_j\ge2, k_j\to\infty$ so that for any $k_j$ and any integer $m\ge 3,$ there exists a smooth Riemannian metric ${g_{j,m}}$ on $M$, so that,
	\begin{align*}
		&\frac{{\mina_\Z^{g_{j,m}}([\Si])}}{\mina_\Z^{g_{j,m}}(k_j[\Si])}\ge m,\\
		&\mina_\Z^{g_{j,m}}(k_j[\Si])= \mina_\R^{g_{j,m}}(k_j[\Si]).
	\end{align*}
\end{thm}
\begin{rem}
	Here $\{k_j\}$ can be any sequence so that $k_j[\Si]$ admit a smoothly embedded representative.
\end{rem}
\begin{rem}
	Thus, it is always possible to fill multiples of a homology class much more efficiently than the class itself.
\end{rem}
The above conclusion cannot be extended to torsion classes as shown in the following theorem.
\begin{thm}\label{ntorsion}
	With the same assumption of Theorem \ref{noncal}, if $[\Si]$ is an $n$-torsion class, $k$ is coprime to and smaller than $n,$ $l$ the smallest natural number with $kl\equiv 1(\mod n),$ then for all metric $g$, we have 
	\begin{align*}
\max\{l\m,(n-l)\m\}\le		\frac{\mina_\Z^{g}(k[\Si])}{\mina_\Z^{g}([\Si])}\le \mina\{k,n-k\}.
	\end{align*}
\end{thm}
\begin{rem}
Here $n$ is the least natural number with $n[\Si]=0.$ 
\end{rem}
\begin{rem}
	The same conclusion holds for $n$-torsion classes in mod $v$ coefficient homology, with $\mina_\Z^g$ replaced with corresponding minimal mass mod $v$.
\end{rem}
Interestingly, the method of proving Theorem \ref{inflav} relies on inspecting area-minimizing currents in finite coefficient homology. To the author's knowledge, this is the first time in the homology setting that mod $v$ mass-minimizing currents are used to study integral and real mass-minimizing currents.
\subsection{Sketch of proof}
Theorem \ref{ntorsion} follows directly from the multiplicative structure of finite abelian groups. Theorem \ref{noncal}, \ref{inflavr} are all straightforward consequences of Theorem \ref{inflav}. For readers only interested in \ref{noncal}, a one sentence summary is as follows. Metrics with the non-calibratable properties in Theorem \ref{noncal} form an open set, so it suffices to find at least one metric with non-calibratable area-minimziers, of which plenty exists due to Theorem \ref{inflav}.

The key idea to prove Theorem \ref{inflav} is that the multiplicative structure of finite abelian groups can sometime forces multiples of a class to admit smaller area.

Let us illustrate the proof for Theorem \ref{inflav} with the simplest possible case, $d=1,M=S^1\times S^c$, $[\Si]=[S^{1}\times \text{point}],$ and $k_1=2.$ Consider the mod $2m-1$ coefficient homology. We have $m(2[\Si])\equiv [\Si]\mod (2m-1).$ In other words, $m$ times the class $2[\Si]$ is class $[\Si]$ in $\Z/(2m-1)\Z$ coefficient. Now take a connected smooth loop $\ga$ representing $2[\Si]$. If one can construct a metric so that $\ga$ is area-minimizing in $2[\Si]$ with $\Z$ coefficient and $m\ga$ is area-minimizing in $2m[\Si]\equiv [\Si]\mod (2m-1)$ in $\Z/(2m-1)\Z$ coefficient, then we are done. The reason is that taking finite coefficients only increases the number of possible competitors, so any $\Z$ coefficient minimizer must have area at least that of $m\ga.$

This minimizing property of $\ga$ and $m\ga$ can be simultaneously achieved by using Zhang's constructions in \cite{YZa} and \cite{YZj}, and using the ideas of Morgan in \cite{FM3}. The connectedness of $\ga$ is essential here. Otherwise, the above argument fails and the minimizer in $2m[\Si]\equiv [\Si]$ might simply be some components of $\ga.$

For Theorem \ref{cpj}, we find a smooth connected algebraic subvariety $N$ representing the class $2[\Si]$. Then we deform the metric arbitrarily close to the Fubini-Study metric while making $N$ uniquely minimizing.

To prove Theorem \ref{singcal}, note that any calibration form must simultaneously calibrate all area-minimizers in the homology class. Thus it suffices to construct a metric where both singular and regular minimizers exist.
	\section*{Acknowledgements}
I cannot thank my advisor Professor Camillo De Lellis enough for his unwavering support while I have been recovering from illness. I feel so lucky that I have Camillo as my advisor. Many thanks goes to him for countless helpful suggestions regarding this manuscript and others. I would also like to thank Professor Michael Freedman, whose personal communications partially inspired this paper. Last but not least, the author wants to thank Professor Frank Morgan for his constant support and pioneering work that has inspired many constructions in the author's works.
\section{Basic Notations}
	In this section, we will fix our notations.
		\subsection{Manifolds and neighborhoods}
	We will reserve $M$ to denote an ambient smooth compact closed orientable Riemannian manifold. Submanifolds will be denoted by $N,L$, etc. We will use the following sets of definitions and notations.
	\begin{defn}\label{defns}
		Let $N$ be a smooth submanifold of a Riemannian manifold $M.$
		\begin{itemize}
			\item $B_r^M(N)$ denotes a tubular neighborhood of $N$ inside $M$ in the intrinsic metric on $M$ of radius $r$. When there is no a priori choice of metric, then use an arbitrary metric.
			\item $\T_p M$ denotes the tangent space to $M$ at $p.$ We will often regard $\T_p N$ as a subspace of $\T_p M.$
			\item $\exp^\perp_{N\s M}$ denotes the normal bundle exponential map of the inclusion $N\s M.$ 
			\item $\fr_N^M$ is the focal radius of $N$ in $M,$ i.e., the radius below which the normal bundle exponential map remains injective.
			\item $\pi_N^M$ denotes the nearest distance/normal bundle exponential map projection from $M$ to $N$ in the metric intrinsic of $M$ in $B_r^N(M)$ with $r$ less than the focal radius of $N$ inside $M$.
		\end{itemize}   
	\end{defn}
	\begin{rem}
		We will often drop the sup/subscripts $M,N$ when there is no confusion.
	\end{rem}
We need the following Lemma about representing homology classes.
\begin{lem}\label{thom}
With the same assumptions as in Theorem \ref{noncal}, if $[\Si]$ is not a torsion class, then there exists a sequence of positive integers $k_j\to\infty$ so that $k[\Si]$ can be represented by a smooth connected embedded orientable submanifold $N.$
\end{lem}
\begin{rem}
	This is the only place we have used the condition $c\ge 2.$
\end{rem}
\begin{proof}
First, let us show that for an integral homology class to have an embedded oriented representative is equivalent to have an embedded connected oriented representative. It is clear that the latter condition implies the former. Let us assume the former condition.

Since the codimension $c$ is larger than $1$, we show that it always possible to use connected sums to connect different components of the representative.
 
	The connected sum is done as follows. Connect any two components using a curve. Then use transversality to make the curve intersecting the two components at only end points. Replace the curve with the sphere bundle in its normal bundle as neck to do the connected sum. Orientation can always be taken into account by twisting the curve around the normal bundle of end points, which is locally always a ball times $\R^2$. Note that in this process both the transversality of the curve and the orientation of the connected sum utilize the condition $c\ge 2.$
	
Thus, we only need to prove the proposition with connectedness condition removed.	Suppose there are only finitely many integers $0<k_1<\cd<k_n,$ so that we have a smoothly embedded representative for $k_j[\Si]$, with $k_n=0$ if there are non. Since $[\Si]$ is not a torsion class, $(k_n+1)[\Si]$ is a nontrivial integral homology class which has no smoothly embedded representative. By Theorem II.29 of \cite{RT}, there exists a non-zero $N,$ so that $N(k_n+1)[\Si]$  has a smoothly embedded representative, and $N(k_n+1)[\Si]$ does not equal $k_j[\Si]$ for all $j$ by non-torsionness. This a contradiction. 
\end{proof}
	\subsection{Currents and real flat chains}\label{cfc}
		When we mention a surface $T$, we mean an integral current $\cu{T}$. For a comprehensive introduction to integral currents, the standard references are \cite{LS1} and \cite{HF}. We will adhere to their notations. Our manuscript mostly focuses on the differential geometric side and in fact, no a priori knowledge of currents is needed. Every time we mention a current, the reader can just assume it to be a sum of chains representing oriented surfaces with singularities. We will use the following definition of irreducibility of currents.
		
		\begin{defn}\label{irr}
			A closed integral current $T$ is irreducible in $U,$ if we cannot write $T=S+W,$ with $\pd S=\pd W=0,$ and $S,W$ nonzero, and $\ms(T)=\ms(S)+\ms(W).$
		\end{defn}
	\begin{rem}
		It is easy to see that the definition of irreducibility above is equivalent to 1.c.i of Section 2.7 of \cite{ZL1}, if we assume the other assumptions in that Section.
	\end{rem}
		Also, we will use the differential geometry convention of closedness. An integral current $T$ is closed, if $\pd T=0$ and $T$ has compact support. 
		
		The primary reference for real flat chains is \cite{HF} and \cite{HF2}. Technically speaking they are the closure of real multiplicity polyhedron chains with finite mass and boundary mass under flat topology (4.1.23 of \cite{HF}). 
		For our purposes,  the only concrete example of real flat chain used is the class of integral currents.
		\begin{rem}
			By Section 3 of \cite{HF2}, the real flat chains in the same real homology class on a compact Riemannian manifold is a closed set. Moreover, there exist homologically mass-minimizing real flat chains in any real homology class.
		\end{rem}
	\subsection{Mass and comass}
	The comass of a $d$-dimensional differential form $\phi$ in a metric $g$ is defined as
	\begin{align*}
		\cms_g \phi=\sup_x\sup_{P\s\T_xM}\phi(\frac{P}{|P|_g}),
	\end{align*}where $P$ ranges over $d$-dimensional oriented planes in the tangent space to $M.$ 

The mass of a current is defined as
\begin{align*}
	\ms_g(T)=\sup_{\cms_g(\phi)\le 1}T(\phi).
\end{align*}
\subsection{Calibrations and real flat cochains}
For calibrations, the reader should be familiar with the definitions of comass and calibrations (Section II.3 and II.4 in \cite{HLg}). The primary reference is \cite{HLg}. The most important concept to keep in mind is the fundamental theorem of calibrations (Theorem 4.2 in \cite{HLg}), which states that calibrated currents are area-minimizing among homologous competitors. We will apply this theorem numerous times without explicit citation. 

Flat cochains are defined as bounded continuous linear functionals on flat chains. By Section 4.1.19 in \cite{FF}, they are equivalent to measurable forms with weak exterior derivatives that both have finite mass.
\begin{defn}\label{cal}
For a flat cochain $\ai,$ we say that $\ai$ is a calibration cochain on a Riemannian manifold $M,$ if 
\begin{itemize}
	\item $d\ai=0,$
	\item $\ai(T)\le \ms(T),$ for all real flat chains $T.$
\end{itemize}
We say a real flat chain $T$ is calibrated by $\ai,$ if
$$\ai(T)=\ms(T).$$	
\end{defn}
\begin{rem}
	By definition of calibration, it is clear that any classical calibration form serves as a calibration cochain.
\end{rem}
We also need to define the singular set of a calibration cochain.
\begin{defn}
	For a calibration flat cochain, the singular set is defined to be the set where all representing weakly closed forms are not continuous.
\end{defn}
\section{Lemmas About Calibrations}
In this section, we need to collect several useful lemmas about calibrations.
\subsection{Basic facts}
\begin{lem}
A real flat chain calibrated by a flat cochain in the sense of Definition \ref{cal} is mass-minimizing in its real homology class.
\end{lem}
\begin{proof}
	This follows directly from Definition \ref{cal}.
\end{proof}
\begin{lem}
	If a real flat cochain $\phi$ calibrates a mass-minimizing flat chain $T$ in a homology class $[T]$, then it calibrates all the mass-minimizing currents in $[T].$ 
\end{lem}
\begin{proof}
	Let $T'$ be any other mass-minimizing chain in $[T].$ Then we have $T'-T=\pd R$ for some real flat chain. Thus we have $(T-T')(\phi)=\pd R(\phi)=R(d\phi)=0.$ This implies $\ms(T)=\phi(T)=\phi(T')\le \ms(T').$ Since $T'$ is also mass-minimizing, we deduce that $\phi(T')=\ms(T').$ 
\end{proof}
\begin{lem}
	For any real homology class $[\Si]$ on a closed compact smooth Riemannian manifold $M,$ every mass-minimizing flat chain in $[\Si]$ is calibrated by some flat cochain $\phi.$
\end{lem}
\begin{proof}
	This is just Section 4.12 of \cite{HF2}.
\end{proof}
\begin{lem}\label{nontors}
	Let $[\Si]$ be a torsion class in integral homology. Then any area-minimizing integral current in $[\Si]$ cannot be calibrated.
\end{lem}
\begin{proof}
	If not, suppose there is a calibration flat cochain $\ai$ that calibrates an area-minimizing integral current $\Si.$ Let $k$ be a non-zero integer so that $k[\Si]=0.$ Then $\ms(T)=\ai(T)=\frac{1}{k}\ai(kT)=0,$ a contradiction.
\end{proof}
\begin{lem}\label{mul}
	Let $T$ be a homologically mass-minimizing real flat chain calibrated by a flat cochain $\ai.$ 
	\begin{itemize}
		\item Then $aT$ is homologically mass minimizing for any real number $a.$ 
		\item If $T$ is the unique real flat chain calibrated by $\ai$ in its real homology class then $aT$ is the unique mass minimizing real flat chain in its homology class for any real number $a.$
	\end{itemize}
\end{lem}
\begin{proof}
If $a$ is zero, there is nothing to prove. From now on, we suppose $a\not=0.$ Note that $$\frac{a}{|a|}\ai(aT)=|a|\ai(T)=|a|M(T)=M(aT).$$ In case of $a>0,$ we deduce the result from $\ai(S)\le \ms(S)$ for any real flat chain $S.$ In case of $a<0,$ note that $-\ai$ is also a calibration cochain in the sense of Definition \ref{cal} and calibrates $-T$. This proves the first claim.

If $T$ is the unique flat chain calibrated by $\ai$ in its real homology class, then it is clearly the unique mass-minimizer by definition. For $a>0,$ suppose $S$ is another mass-minimizing real flat chain $\ai$ in the homology class $a[T],$ then we have
\begin{align*}
	\frac{\ms (aT)}{a}=\frac{\ms(S)}{a}=\ms(S/a)\ge \ai(S/a)=\ai(T)=\ms(T).
\end{align*}
Thus, we have $\ms(S/a)=\ai(S/a).$ This implies $S=aT.$ If $a<0,$ we run the argument with $-\ai$ replacing $\ai$ and $-T$ replacing $T.$ 
\end{proof}
\subsection{Continuity of mass}
\begin{lem}\label{msbd}
For any current $T$ and two different smooth metrics $g,h,$ we have
\begin{align*}
	\no{g/h}_{C^0}^{-d/2}\ms_g(T)\ge\ms_h(T)\ge \no{h/g}_{C^0}^{d/2}\ms_g(T).
\end{align*}
\end{lem}

\begin{rem}
	Here $\no{g/h}_{C^0}$ is defined as
	\begin{align*}
		\no{g/h}_{C^0}=\sup_{x\in M}\sup_{v\in \T_x M}\frac{g(v,v)}{h(v,v)}.
	\end{align*}
\end{rem}
\begin{proof}
	By symmetry of $g$ and $h,$ it suffices to show only the inequality on the right hand side.
	Let $c=\no{h/g}_{C^0}.$ Then for any $d$-dimensional differential form $\phi$, we have
	\begin{align*}
		\cms_h(\phi)=\sup_x\sup_{P}\phi_x\bigg(\frac{P}{|P|_h}\bigg)\ge \sup_x\sup_{P}\phi_x\bigg(\frac{P}{c^{d/2}|P|_g}\bigg)=\frac{1}{c^{d/2}}\cms_g(\phi).
	\end{align*}where $P$ runs through all $d$-dimensional oriented planes. This implies that $$\{\phi|\cms_g(\phi)\le c^{d/2}\}\s\{\phi|\cms_h(\phi)\le 1\}.$$ Since $\sup$ of a subset is always smaller than or equal to that of the whole set, we deduce that, 
\begin{align*}
\ms_h(T)=&\sup_{\cms_h(\phi)\le 1}T(\phi)\ge \sup_{\cms_g(\phi)\le c^{d/2}}T(\phi)\\=&\sup_{\cms_g(c^{-d/2}\phi)\le 1}c^{d/2}T(c^{-d/2}\phi)=c^{d/2}\ms_g(T).
\end{align*}
\end{proof}
\begin{lem}\label{convms}
Let $g_j\to g$ be a sequence of smooth metrics $\{g_j\}$ converging to $g.$ For any sequence of mass-minimizing real flat chain $T_j\in [\Si]$ in the sequence of metric $g_j$, any converging subsequence converges to a mass-minimizing flat chain in $g$, and we have
\begin{align*}
	\lim_j \ms_{g_j}(T_j)=\ms_g(T),
\end{align*}
where $T$ is any mass-minimizing flat chain in metric $g$ and $\ms_{g_j}$ denote the mass in $g_j.$
\end{lem}
\begin{proof}
By Section 3.7 of \cite{HF2}, the set $\{T_j\}$ is pre-compact with respect to $g.$  Let $\{T_{j_k}\}$ be any converging subsequence and $T'$ its limit. Now apply the smoothing deformation (Section 4.7 of \cite{HF2}) to $T'-T_{j_k}.$ By Section 4.7 (mainly page 377), for any $\e>0,$ there exists $N_\e,$ so that for all $k\ge N_\e,$ we have $T'-T_{j_k}=S_{j_k}+\pd R_{j_k},$ with $\ms_g(S_{j_k}),\ms_g( R_{j_k})\le \e,$ and $S_{j_k},R_{j_k}$ also being real flat chains. Writing $T'=S_{j_k}+T_{j_k}+\pd R_{j_k}$, using triangle inequality and mass-minimizing property of $T_{j_k}$, we have
\begin{align*}
\ms_{g_{j_k}}(T')\ge\ms_{g_{j_k}}(T_{j_k}+\pd R_{j_k})-\ms_{g_{j_k}}(S_{j_k})\ge \ms_{g_{j_k}}(T_{j_k})-\ms_{g_{j_k}}(S_{j_k}).
\end{align*}This implies $\liminf\ms_{g_{j_k}}(T')\ge\limsup \ms_{g_{j_k}}(T_{j_k}).$
Now use Lemma \ref{msbd}, we deduce that
\begin{align}\label{gliminf}
\ms_g(T')\ge \limsup \ms_{g_{j_k}}(T_{j_k}).
\end{align}
 On the other hand, by lower-semicontinuity of mass, e.g., Section 2.8 of \cite{HF2}, we have $$\ms_g (T')\le\liminf \ms_{g}(T_{j_k}).$$ By Lemma \ref{msbd}, this gives \begin{align}\label{glimsup}
 	\ms_g(T')\le\liminf \ms_{g_{j_k}}(T_{j_k}).
 \end{align}
Combining (\ref{gliminf}) and (\ref{glimsup}) gives $\lim_{j_k}\ms_{g_{j_k}}(T_{j_k})=\ms_g(T')$.

If we can prove $T'$ is mass-minimizing, then we are done, since all mass-minimiers in the same homology class must have the same mass. To see this, we have
\begin{align*}
\ms_{g_{j_k}}(T'+\pd W)=&\ms_{g_{j_k}}(S_{j_k}+T_{j_k}+\pd R_{j_k}+\pd W)
\\\ge& - \ms_{g_{j_k}}(S_{j_k})+\ms_{g_{j_k}}(T_{j_k}+\pd R_{j_k}+\pd W)\\\ge& - \ms_{g_{j_k}}(S_{j_k})+\ms_{g_{j_k}}(T_{j_k}).
\end{align*}
Take limit on both sides and use Lemma \ref{msbd} gives $\ms_g(T'+\pd W)\ge \ms_g(T').$
\end{proof}
\begin{rem}
	Replacing the smoothing deformation with deformation theorem of integral currents, the above argument shows that the mass is continuous on area-minimizing integral currents in sequences of metrics.
\end{rem}
\subsection{Non-calibratable condition}
\begin{lem}\label{calcond}
An area-minimizing integral current $T$ in $[\Si]$ is calibrated by a flat cochain if and only if $\ms(T)=\mina_\R([T])$.
\end{lem}
\begin{proof}	
	If $T$ is calibrated by a flat cochain $\phi$, then we have
	\begin{align*}
		\ms(T)=\phi(T)=\phi(Z)\le \ms(Z).
	\end{align*}
	By mass-minimality of $Z$, we deduce that $\ms(T)=\ms(Z).$
	
	If $\ms(T)=\ms(Z)$, then by 4.12 in \cite{HF2}, there exists a calibration flat cochain $\psi$ with $\lam(Z)=\ms(Z).$ We have $\ms(T)\ge\lam(T)=\lam(Z)=\ms(Z).$ Since $\ms(T)=\ms(Z),$ we deduce that $\ms(T)=\lam(T).$ Thus $T$ is calibrated by $\lam.$	
\end{proof}
\begin{lem}\label{open}
Let $M$ be a smooth Riemannian manifold $M$ with a metric $g$ and $[\Si]$ a $d$-dimensional integral homology class.  If no area-minimizing integral current in $[\Si]$ can be calibrated in $g,$ then there exists an open set $\Om_{[\Si]}$ containing $g$, so that for any metric $g'\in\Om_{[\Si]},$ no area-minimizing integral current in $[\Si]$ can be calibrated by a flat cochain.
\end{lem}
\begin{proof}
By Lemma \ref{calcond} if no area-minimizing integral current in $[\Si]$ in $g$ can be calibrated, we must have 
\begin{align*}
	\inf_{\text{real flat chain }T\in[\Si]} \ms_g(T)<\inf_{\text{integral current }T\in[\Si]}\ms_g(T).
\end{align*}
However, by Lemma \ref{convms} and the remark below, both sides of the inequalities are continuous functions with respect to the metric $g$. This finishes the proof.
\end{proof}
\subsection{Miscellaneous facts}
\begin{lem}\label{singco1}
	The singular set of any real flat cochain calibrating an area-minimizing hypersurface contains the singular set of all area-minimizing hypersurfaces in the same class.
\end{lem}
\begin{proof}
	Theorem 4.2 in \cite{VB}.
\end{proof}
\begin{lem}\label{comp}
	Consider $\R^{2n}=\C^n,$ with standard metric and coordinate $$(x_1,y_1,\cd,x_n,y_n).$$ Define
	\begin{align*}
		\psi=\sum_j a_jQ_j\du,
	\end{align*}with $a_j$ real numbers, $Q_j$ holomorphic $l$-dimensional planes spanned by the coordinate axes, and $\ast$ the musical isomorphism. Then we have
\begin{align*}
	\cms\psi=\max_j |a_j|.
\end{align*}
Moreover, if $|a_j|<1$ for all $j\ge 2$ and $a_1=1.$ Then $\psi$ calibrates only the plane $Q_j.$
\end{lem}
\begin{proof}
The comass equality is just Theorem 2.2 in \cite{DHM}. To show the claim about calibrating only $Q_j,$ we prove inductively on $l.$ If $l=1,$ we write
\begin{align*}
	\psi=Q_1\du+\sum_{j\ge 2}Q_j\du.
\end{align*}By Theorem 2.2 in \cite{DHM}, the sum part has comass $\max_{j\ge 2}|a_j|<1.$ Thus, only alternative (i) of Lemma 2.1 of \cite{DHM} can happen, and we deduce that the comass $1$ is only achieved at $Q_1.$ Suppose this is true for $l=1,\cd,k-1.$ Now suppose $l=k.$ Without loss of generality, we can suppose that $Q_1$ is spanned by $dx_1,dy_1,\cd,dx_l,dy_l.$ We sort out the situation into two possible cases.

One is that $Q_1$ does not intersect any other $Q_j.$ Then argue as in the $l=1$ case, we are done.

The other is that $Q_1$ does intersect some other $Q_j$. Without loss of generality, suppose $Q_1\cap Q_2$ includes the span of $x_1,y_1$ axes. Then we can factor out all $dx_1dy_1$ terms, i.e., writing
\begin{align*}
	\psi=dx_1dy_1\w(dx_2dy_2\w\cd\w dx_ldy_l+\sum_j a_j\ti{Q_j}\du)+\sum_l a_l (Q_l')\du.
\end{align*}
Here $\ti{Q_j}$ are $l-1$-dimensional complex planes spanned by coordinate axes not including $x_1,y_1$ and $Q_l'$ are $l$-dimensional complex planes spanned by coordinate axes not including $x_1,y_1.$

The second sum has comass smaller than $1$ by Theorem 2.2 of \cite{DHM}. Thus by Lemma 2.1 of \cite{DHM}, only alternative (i) of that Lemma can happen. In other words, if $Q$ is a plane calibrated by $\psi,$ then it is calibrated by $dx_1dy_1\w(dx_2dy_2\w\cd\w dx_ldy_l+\sum_j a_j\ti{Q_j}\du).$ By Proposition 7.10 in \cite{HLg}, this implies that $Q$ is the product of $x_1y_1$-plane and a plane $Q'$ calibrated by $dx_2dy_2\w\cd\w dx_ldy_l+\sum_j a_j\ti{Q_j}\du$. By inductive hypothesis, we are done.
\end{proof}
\subsection{Zhang's constructions}
\begin{lem}\label{zhang}
	With the same assumptions as in Theorem \ref{noncal}, suppose $[\Si]$ is non-torsion and admits a connected embedded representative $N$. Let $m$ be any positive integer. Then there exists a smooth Riemannian metric $\hat{g}$ on $M$, so that.
	\begin{itemize}
			\item there is a smooth neighborhood $W$ around $N,$ and a retract $\pi_N:W\to N,$ so that $\pi_N$ is an area-non-increasing in $\hat{g},$
			\item $N$ has area precisely $1,$
			\item any stationary varifold in $M$ whose support is not contained in $W,$ has area larger than $t,$ with $t$ an arbitrary positive real number,
					\item $N$ is calibrated by a smooth form in $M.$
	\end{itemize}
\end{lem}
\begin{proof}
	Equip $M$ with an arbitrary smooth metric. Take $B_r(N)$ to be a tubular neighborhood of $N$. Use Lemma 3.4 and Remark 3.5 of \cite{YZa}. We get a retract $\pi_N$ and a smooth metric $\ov{g}$ on $B_r(N)$ so that $\pi_N\du(dvol_N)$ has comass at most. Now apply Lemma 5.1 of \cite{ZL1}, we see that $\pi_N$ is area-non-increasing.
	
	Since $[N]$ is a non-torsion class, by Section 3.3 of \cite{YZa}, there exists a smooth metric $\ti{g}$ and a closed $d$-dimensional calibration form $\Phi,$ so that $\Phi$ calibrates $N$. Moreover, $\Phi$ and $\ti{g}$ restricted to $B_{\frac{3}{5}r}(N)$ (here $B$ is measured in $\ti{g}$) equal $\pi_N\du(dvol_N)$ and $\ov{g},$ respectively. 
	
	Now take $W=B_{\frac{3}{5}r}(N)$ and apply the proof of Theorem 4.1 of \cite{YZj} to get a new metric $\hat{g}$. A constant scaling of $\Phi$ remains a calibration that calibrates $N.$ Moreover, any stationary varifold not contained in $W$ has area larger than $t$ times the area of $N,$ with $t$ an arbitrarily large constant. Just rescale the metric so that $N$ has area $1$. We are done.  
\end{proof}
\section{Proof of Theorem \ref{inflav}}
By Lemma \ref{thom}, there exists a sequence of integers larger than $1 $, $k_j\to\infty,$ so that $k_j[\Si]$ can be represented by a smooth connected embedded orientable submanifold. The proof for different $k_j$ and $m$ is the same, so we fix one $j$ and one $m$ and suppose the connected embedded oriented representative is $N.$

By Lemma \ref{zhang} and Lemma \ref{calcond},  there exists a smooth metric $\hat{g}$, so that  
\begin{enumerate}
	\item there is an area-non-increasing retract $\pi_N:W\to N$ from a neighborhood $W$ of $N$ onto $N,$
	\item $N$ realizes $\mina_\Z^{\hat{g}}(k_j[\Si])$, and has area $1$,
	\item any stationary varifold not contained in $W$ has area larger than $m.$
	\item $\mina_\Z^{\hat{g}}(k_j[\Si])=\mina_\R^{\hat{g}}(k_j[\Si])$. 
\end{enumerate}
There exists a largest integer $l,$ so that $[\Si]=l[\Si_0]$ for some non-torsion integral homology class $[\Si_0].$

Consider the mod $(mk-1)l$ homology, in other words, homology over coefficient of $\Z/(mk-1)l\Z.$ By the universal coefficient theorem \cite{AH}, $H_d(M)\ts \Z/(mk-1)l\Z$ cannonically injects into $H_d(M,\Z/(mk-1)l\Z)$. Use $I$ to denote this map.

Then $$mI(k[\Si])=mklI([\Si_0])=lI([\Si_0])=I([\Si]).$$ Moreover, by $m\ge 3,k\ge 2,$ $I(k[\Si])$ and $I([\Si])$ both are non-zero classes.

Now let $T$ be a mass-minimizing mod $(mk-1)l$ current in the class $I([\Si]).$ Then $T$ is non-trivial. By bullet point (3) above, if $T$ is not contained in $W$, then we are done. If $T$ is contained in $W$, then a straightforward adaptation of Theorem 2.2. of \cite{FM3} and the fact that $m\le \frac{mk-1}{2}$ show that $mN$ is an area-minimizing current mod $(mk-1)l$ in $W.$ However, the homology of $W$ is generated by $N,$ so we deduce that $T$ must be homologous to $mN$ in $W$ mod $(mk-1)l.$ This implies that $T$ has the same mass as $mN,$ namely $m.$ Since mod $(mk-1)l$ minimizing adds more competitors, we deduce that any area-minimizing integral current in the class $[\Si]$ must have area at least that of $T.$ We are done.
\section{Proof of Theorem \ref{noncal},\ref{inflavr},\ref{ntorsion}}
\subsection{Proof of Theorem \ref{inflavr}}
By Theorem \ref{inflav}, for $k_1\ge 2,$  and any integer $m\ge 3,$ there exists a smooth Riemannian metric $g_{1,m}$ so that $$\mina_\Z^{g_{1,m}}(k_1[\Si])= \mina_\R^{g_{1,m}}(k_1[\Si]), \frac{{\mina_\Z^{g_{1,m}}([\Si])}}{\mina_\Z^{g_{1,m}}(k_1[\Si])}>m.$$

This implies that $$\mina_\R^{g_{1,m}}([\Si])=\frac{1}{k_1}\mina_\R^{g_{1,m}}(k_1[\Si])=\frac{1}{k_1}\mina_\Z^{g_{1,m}}(k_1[\Si])<\frac{1}{m}{\mina_\Z^{g_{1,m}}([\Si])}.$$
Let $m\to\infty.$ We are done.
\subsection{Proof of Theorem \ref{noncal}}First of all, by Lemma \ref{nontors}, we only have to consider non-torsion classes, i.e., $n[\Si]\not=0$ for any $n>0.$ 
Then the assertion follows directly from Lemma \ref{calcond}, Lemma \ref{open} and Theorem \ref{inflavr}.
\subsection{Proof of Theorem \ref{ntorsion}}
If $[\Si]$ is $n$-torsion, then we have $k[\Si]=-(n-k)[\Si].$ Note that the multiplication by $-1$ map sends integral currents in $[\Si]$ bijectively to integral currents in $-[\Si]$. On the other hand, for any integer $n$, we always have $\ms (nT)=|n|\ms (T)$ directly from definition.
Thus, we have 
\begin{align*}
\mina_\Z^g(k[\Si])=&\mina_\Z^g(-(n-k)[\Si])=\mina_\Z^g((n-k)[\Si]),\\
\mina_\Z^g(k[\Si])=&\inf_{S\text{ homologous to }k[\Si]\text{ over }\Z} \ms(S)\le \inf_{S\text{ homologous to }[\Si]\text{ over }\Z}\ms(kS)\\=&k\mina_\Z^g([\Si]),
\end{align*}
and similarly 
$$\mina_\Z^g((n-k)[\Si])\le (n-k)\mina_\Z^g([\Si]).$$
Combining the three gives
\begin{align}\label{right}
	\frac{\mina_\Z^g(k[\Si])}{\mina_\Z^g([\Si])}\le \min\{k,n-k\}.
\end{align}
On the other hand, if $kl\equiv 1(\mod n),$ then
\begin{align*}
	l(k[\Si])=[\Si].
\end{align*}
Apply (\ref{right}) with $k[\Si]$ replacing $[\Si]$, $l$ replacing $k,$ $[\Si]$ replacing $k[\Si]$, we are done.
\section{Proof of Theorem \ref{cpj}}
Note that $2[\Si]$ admits a smoothly embedding connected subvariety $N.$ This can be done in many ways. For instance, suppose $[\Si]$ is $k$ times the generator of $d$-dimensional homology on $\cpt^{\frac{d+c}{2}}$. Then consider the inclusion of subspaces $\cpt^{\frac{d}{2}+1}\s\cpt^{\frac{d+c}{2}}.$ It sends any degree $2k$ projective hypersurface of the former into the class $2[\Si]$ of the latter, e.g., by Mayer-Vietoris. A connected degree $2d$ hypersurface in $\C^{\frac{d}{2}+1}$ can be constructed explicitly, e.g., the Fermat hypersurfaces.

Now if we consider the corresponding affine variety back in $\C^{\frac{d+c}{2}+1},$ we have a complex algebraic cone $C$ corresponding to $N.$ Let $\pi_C$ be the nearest distance projection onto $C.$ Then by construction, the orthogonal complement to $\ker\pi_C$ is a smooth distribution of complex $\frac{d}{2}$ dimensional planes. Use
\begin{align*}
	\ti{\Ga}
\end{align*} to denote this distribution on $\C^{\frac{d+c}{2}+1}$.
 The distribution is invariant under diagonal action of $U(1)$ and scalings by real parameters, as $C$ and the metric on $\R^n$ is up to scalings. Note that the complex structure and Fubini-Study metric on complex projective space can be defined via the projection from the unit sphere e.g., Section 6.5 of \cite{AG}.
 
 Thus by projecting onto the complex projective space, we deduce a smooth distribution $$\Gamma$$ of complex $\frac{d}{2}$-dimensional planes in a tubular neighborhood $B_{2r}(N)$ of $N,$ so that $\Ga|_N$ equals the tangent space to $N.$
\subsection{New metric}
We need to construct a sequence of new metrics on $\{g_j\}.$ Here the notations are as follows.
\begin{itemize}
	\item $g_{FS}$ is the Fubini-Study metric,
	\item $g_{\Ga}$ is $g_{FS}$ restricted to $\Ga,$
	\item $g_{\Ga^\perp}$ is $g_{FS}$ restricted to the orthogonal complement to $\Ga,$
	\item $\textnormal{dist}(p,N)$ is the distance of $p$ to $N,$
	\item $\be$ is an even function monotonically increasing on $[0,\infty)$, equal to $(1+\frac{1}{j})$ on $[\frac{4}{9}r,\infty)$ and equal to $1$ only at $0.$
\end{itemize}
\begin{defn}
	\begin{align*}
		g_j(p)=\begin{cases}
			(1+\frac{1}{j})g_{FS}(p)\text{ on }B_r(N),\cp\\
			(1+\frac{1}{j})g_{\Gamma^\perp}(p)+\be(\textnormal{dist}(p,N)^2)g_{\Ga}(p)\text{ on }B_{{r}}(N)
		\end{cases}
	\end{align*}
	
\end{defn}
It is straightforward to verify that $g_j$ converges smoothly to $g_{FS}$ as $j\to\infty.$
\begin{lem}
	The $\frac{d}{2}$ power of K\"{a}ler form $\om_{d/2}=\frac{1}{(d/2)!}\om^{d/2}$ is still a calibration in $g_j$ and calibrates only $N$ in $2[\Si].$
\end{lem}
\begin{proof}
$\om_{d/2}$ is a calibration on $B_r(N)\cp$ with comass $(1+\frac{1}{j})^{-\frac{d}{2}}$ by Theorem 2.2 of \cite{DHM}. For any point $p\in B_r(N),$ the tangent space $\T_p$ to $\cpt^{\frac{d+c}{2}}$ admits a splitting into orthogonal complex subspaces $\Ga$ and $\Ga^\perp.$ In $g,$ taking an orthonormal basis $\{e_1,Je_1,\cd,e_{d/2},Je_{d/2}\}$ of $\Ga$ and $\{f_1,Jf_1,\cd,f_{c/2},Jf_{c/2}\}$ of $\Ga^\perp$. Then we have
\begin{align*}
	\om_{d/2}=\frac{1}{(\frac{d}{2})!}\big(\sum_le_l\du\w (Je_l)\du+\sum_mf_m\du\w(Jf_m)\du\big)^{d/2},
\end{align*}
where $\ast$ denotes the musical isomorphism in $g.$ Now we reinterpret the above expression in the rescaled orthonormal basis in $g_j$. 
On $B_{{r}}(N),$ we have
\begin{align*}
	\om_{d/2}=&\frac{1}{(\frac{d}{2})!}\bigg(\be\m\sum_l\be^{1/2}e_l\du\w \be^{1/2}(Je_l)\du\\&+(1+\frac{1}{j})\m\sum_m(1+\frac{1}{j})^{1/2}f_m\du\w(1+\frac{1}{j})^{1/2}(Jf_m)\du)^{d/2}.
\end{align*}
By Lemma \ref{comp}, we deduce that
\begin{align*}
	\cms_{g_j}\om_{d/2}=\max\{\be^{-\frac{d}{2}},(1+\frac{1}{j})^{-\frac{d}{2}}\},
\end{align*}which is clearly at most $1$.

 Lemma \ref{comp} shows that $\om_{d/2}(Q)=1$ for a unit simple $d$-vector $Q$ in $B_{r}(N)$ if and only if $Q\in\Ga$ and $\be=1.$ This implies $Q$ is tangent to $N$. By the constancy theorem (4.1.7 in \cite{HF}), integer multiples of $N$ are the unique integral currents calibrated by $\om_{d/2}.$ Since $2[\Si]$ is not a torsion class, we deduce that $\om_{d/2}$ only calibrates $N$ in $2[\Si].$
\end{proof}
Since $N$ is irreducible by Lemma 2.10 in \cite{ZL1}, arguing as in previous sections show that no area-minimizing current in $[\Si]$ can be calibrated in $g_j.$ By Lemma \ref{open}, we are done.
\section{Proof of Theorem \ref{singcal}}
By Lemma \ref{singco1}, it suffices to construct a metric with at least one smooth minimizer and at least one singular minimizer.

By representation theorem of \cite{WM}, there exists a primitive class $[\ga],$ so that $\ga$ can be represented by a smoothly embedded minimal hypersurface $N$ and $[\Si]=k[\ga]$ for some $k>0.$

Since the normal bundle to $N$ is trivial, we can flow $N$ to one side and get another embedded representative $N'.$ Let $C$ be any $7$-dimensional area-minimizing hypercone in $\R^8.$ Consider the following hypersurface 
\begin{align*}
	\si(C)=(C\times \R^{(d-7)+1})\cap S^{8+(d-7)}\s \R^{8+(d-7)+1}.
\end{align*} 
By Lemma 4.1 in \cite{ZL1}, $\si(C)$ is smooth outside of $0\times S^{d-7}\s\R^{8}\times \R^{(d-7)+1},$ and near the singular set $\si(C)$ can be sent to truncated $C$ times $S^{d-7}$ by a diffeomorphism. Note that by Theorem 2.1 in \cite{RHLS}, there is a foliation of $\R^8$ by area-minimizing hypersurfaces with  $C$ as one leaf. Let $\nu$ be the normal vector field of the foliation then $\phi=\ast\nu^\ast$ will be a calibration form with an isolated singularity at $0.$ Here the first $\ast$is  the Hodge star and the second $\ast$ is the musical isomorphism. By taking $\phi\w dvol_{S^{d-7}}$ and taking a product metric, we get a calibration form that calibrates $\si(C)$ near its singular set.

Now pick a point $p$ in $N'$. Then we can embed $\si(C)\s S^{8+(d-7)}-\text{pt}$ into the upper half ball of $B_{r}(p)$ for $r<\frac{1}{2}d(N,N'.)$ Then we can do a connected sum to connect $\si(C)$ to $N'.$ Note that $N'\#\si(C)$ still possess the same homology class as $N.$

By Lemma 2.9,2.8,2.7 (in this order) of \cite{ZL1}, there exists a smooth metric $g$, so that both $N'\# \si(C)$ and $N$ are homologically area-minimizing in $[\ga],$ the primitive class associated with $[\Si].$ 

By Section 5.10 of \cite{HF2} and Lemma \ref{mul}, we deduce that $k(N'\#\si(C))$ and $kN$ are both mass-minimizing real flat chains in $[\Si].$ By Lemma \ref{singco1}, we are done.

\end{document}